\documentclass[12pt]{article}
\usepackage{amsmath,amsthm,amssymb,hyperref, color}
\usepackage{fullpage}

\newtheorem{thm}{Theorem}[section]
\newtheorem{claim}[thm]{Claim}

\newtheorem{define}[thm]{Definition}

\newtheorem{com}[thm]{Comment}

\newtheorem{conjecture}[thm]{Conjecture}

\def\GL{{\mathbf{GL}}}
\def\e{{\mathbf{e}}}
\def\ord{{\mathbf{ord}}}

\def\F{{\mathbb{F}}}

\def\N{{\mathbb{N}}}

\def\I{{\mathbf{I}}}

\newcommand{\ip}[2]{\langle #1,#2 \rangle}

\def\_{\,\,\,\,\,}

\def\id{ \textit{id} }

\def\supp{\textsf{supp}}

\def\rank{\textsf{rank}}
\def\poly{\textsf{poly}}

\newcommand{\eps}{\epsilon}

\newcommand{\remove}[1]{}

\begin{document}

\title{ Rank-metric separation in irreducible representations of finite groups}

\author{Zeev Dvir\thanks{Department of Mathematics  and Department of Computer Science,
Princeton University.
Email: \texttt{zdvir@princeton.edu}. Research supported by NSF grant DMS-2246682.}}

\date{}
\maketitle

\begin{abstract}
We give a general lower bound on the rank of matrices of the form $\rho(h) - \I$ with $\rho : G \rightarrow \GL(\F^n)$ an irreducible representation of a finite group $G$. The main tool in the proof is a (strengthening) of a reduction due to  Efremenko from low rank  matrices spanned by a few images of $\rho$ to Locally Decodable Codes (LDCs), which are a special kind of error correcting codes. We then apply the known results on 2-query LDCs to derive our rank bound.
\end{abstract} 


\section{Introduction}

Let $\F$ be a field, $G$  a finite group and let $\rho : G \rightarrow \GL(\F^n)$ be an irreducible  representation. We let 
$$ \theta_\F =  \begin{cases}
1 & \text{if } \F \text{ is infinite} \\
1 - \frac{1}{|\F|} & \text{otherwise}
\end{cases}.$$
And, for any $h \in G$ we define
$$ \gamma_h = \begin{cases}
1 & \text{if } \ord(h) \text{ is even} \\
1 - \frac{1}{\ord(h)} & \text{otherwise}
\end{cases}.$$

Our main result is the following.
\begin{thm}\label{thm-main2LDCsep}
For all $h \in G$ such that $\rho(h) \neq \I$ we have
$$ \rank(\rho(h) - \I ) \geq \frac{\theta_\F \cdot \gamma_h \cdot  n}{\log_2|G|} \geq \frac{n}{3\log_2|G|}$$
\end{thm}

	The following construction shows the tightness of this bound: Take the group of $n \times n$ matrices generated by (i) all diagonal matrices with diagonal entries in $\{1,-1\}$ (suppose $\text{char}(\F) \neq 2$) and (ii) all cyclic shifts of the $n$ coordinates. This group has size $n2^n$ and its action on $\F^n$ is irreducible. However, the rank of $\rho(h) - I$ can be as low as 1 for the diagonal matrix with exactly one negative entry on the diagonal (a reflection).

Another way to state this theorem is as an upper bound on the dimension of the fixed space $$C(h) = \{ v \in \F^n \,|\, \rho(h)v = v\}$$ of the form
$$ \dim( C(h) ) \leq \left(1 - \frac{\theta_\F \gamma_h}{\log_2|G|}\right)n.$$  This complements a result of Guralnick and Mar\'oti who showed in \cite{GurMar11} a much stronger upper bound for the {\em average} dimension of $C(h)$: 
$$ \frac{1}{|G|} \sum_{h \in G} \dim( C(h) ) \leq \frac{n}{2},$$ proving a conjecture of Neumann \cite{neumann1966a}.

The proof of Theorem~\ref{thm-main2LDCsep} can also be used to bound the rank distance to a  scalar multiple of the identity, showing that
\begin{equation}\label{eq-rankwithlambda}
	 \rank(\rho(h) - \lambda \I ) \geq \frac{\theta_\F \cdot \gamma_h \cdot  n}{2\log_2(2|G|)},
\end{equation}
 for any $\lambda \in \F$ such that $\rho(h) \neq \lambda \I$ (see Comment~\ref{com-lambda} at the end of Section~\ref{sec-rep2LDC}). When $G$ is not too large (specifically, when $G \ll 2^{\sqrt{n}}$) this bound improves on a theorem of Guralnick and Saxl \cite{GurSax03} which gives an lower bound of $\sqrt{n}/2$ for the same quantity.

The proof of Theorem~\ref{thm-main2LDCsep} is via a connection, originally given by Efremenko \cite{Efr12}, between representations and a type of error correcting codes called Locally Decodable Codes (LDCs). To formally define these codes we need the notion of a $q$-matching.

\begin{define}[$q$-Matching]
Let $n,q \geq 1$ be integers. A {\em $q$-matching} over $[n]$ is a family of subsets $M \subset 2^{[n]}$ such that 
\begin{enumerate}
	\item For every $S \in M$, $|S|=q$.
	\item For every $S,T \in M$, $S \neq T$ we have $S \cap T = \emptyset.$
\end{enumerate}
The {\em size} of a $q$-Matching  $M$ is the number of sets in $M$.
\end{define}

Locally Decodable Codes were first defined explicitly by Katz and Trevisan \cite{KatzTre00} (but appeared implicitly in earlier works). For our purposes it is enough to consider linear  $LDC$s. 
\begin{define}[Linear $(q,\delta)$-LDC]\label{def-LDC}
Let $\F$ be a field, $q,n,m \geq 1$ integers and $\delta \in [0,1]$. A sequence $A = (a_1,\ldots,a_m) \in (\F^n)^m$ of $m$ vectors in $\F^n$ is a linear {\em $(q,\delta)$-Locally Decodable Code (LDC)} if, for all $i \in [n]$ there exists a $q$-Matching $M_i$ over $[m]$ such that
\begin{enumerate}
	\item For every $S \in M_i$ we have that $\e_i$, the $i$'th standard basis vector, is spanned by the set $\{ a_j \,|\, j \in S\}$.
	\item  $\sum_{i=1}^n |M_i| \geq \delta m n$.
\end{enumerate}
\end{define}
It is more common to define LDCs as linear encodings (computed by a matrix whose rows are the $a_i$'s) admitting a probabilistic `local decoding' procedure, which can recover the  $i$'th coordinate of a message from a possibly corrupted encoding using at most $q$-queries. For our purposes, however, the above definition will be easier to work with  and it can be shown that the two definitions are equivalent up to constants for linear codes (see \cite{DviShp07} for example). 

One should, in general, think of $q$ and $\delta$ as fixed constants and then ask about the dependence between $m$ and $n$ as they tend to infinity, where the main `goal' is to either construct codes with $m$ as small as possible (ideally polynomial in $n$) or show that such constructions are not possible.\footnote{Unlike with normal error correcting codes, there are no probabilistic constructions showing existence of LDCs.} We sometimes omit the parameter $\delta$  and refer to a code as a $q$-LDC.

For the case $q=2$, which is our main interest, it will be convenient to work with the following, slightly more restrictive definition. In this `special' form  of $2$-LDCs we require the pairs in each matching to span the corresponding standard basis vector as a multiple of their difference (essentially, we are asking that the two vectors differ only in the $i$'th coordinate).\footnote{One can  again show that the two definitions are equivalent up to constants \cite{DviShp07}.}
\begin{define}[Special  $(2,\delta)$-LDC]\label{def-special2ldc}
Let $\F$ be a field, $n,m \geq 1$ integers and $\delta \in [0,1]$. A sequence $A = (a_1,\ldots,a_m) \in (\F^n)^m$ of $m$ vectors in $\F^n$ is a `special' linear {\em $(2,\delta)$-Locally Decodable Code (LDC)} if, for all $i \in [n]$ there exists a $2$-Matching $M_i$ over $[m]$ such that
\begin{enumerate}
	\item For every $S = \{j_1,j_2\} \in M_i$ there is a scalar $\lambda \in \F$ such that $\e_i = \lambda(a_{j_1} - a_{j_2})$.
	\item $\sum_{i=1}^n |M_i| \geq \delta m n$.
\end{enumerate}

\end{define}

Efremenko, in \cite{Efr12}, proved that $q$-LDCs can be constructed from an irreducible representation $\rho : G \rightarrow \GL(\F^n)$ whenever there exists $q$ group elements $h_1,\ldots,h_q$ such that the matrices $\rho(h_i)$ span a rank one matrix. We are able to slightly modify Efremenko's proof  in a way that  allows the rank to be larger than one, obtaining the following theorem.
\begin{thm}\label{thm-rep2LDC}
Suppose that there are $q$ group elements $h_1,\ldots,h_q \in G$ and $q$ scalars $\alpha_1,\ldots,\alpha_q \in \F$ such that $$ 0 < \rank\left(\sum_{\ell=1}^q \alpha_\ell \rho(h_\ell)\right) \leq R.$$  Then, there exists a linear $(q,\delta)$-LDC $A = (a_1,\ldots,a_m) \in (\F^t)^m$ with 
$$m = |G|, \,\, t = \lceil n/R \rceil, \text{ and } \delta =\theta_\F/q^2.$$
\end{thm}

For $q=2$ we can w.l.o.g assume that $h_2 = \id$. The proof of Theorem~\ref{thm-rep2LDC} then gives the following, slightly more refined bound for matrices of the form $\rho(h) - \I$.
\begin{thm}\label{thm-repLDC-2q}
 Let $h \in G$ be such that
 $$0 <\rank\left(\rho(h) - \I \right) \leq R. $$
Then there exists a special $(2,\delta)$-LDC $A = (a_1,\ldots,a_m) \in (\F^t)^m$ with 
$$m = |G|,\,\, t = \lceil n/R \rceil \text{ and }\,\, 
\delta = \frac{\theta_\F \gamma_h}{2}$$
\end{thm}

We note here that the proof of these theorems actually shows that  the sequence $A$ can be taken to be a projection of a $\rho$-orbit. That is, there exists a linear map $P : \F^n \rightarrow \F^t$ and a vector $z \in \F^n$ such that for all $j \in [m]$, $a_j = P(\rho(g_j) z)$, where $G = \{g_1,\ldots,g_m\}$ is some ordering of $G$ (in fact, any $z$ avoiding a given  finite set of hyperplanes will do).

We will derive Theorem~\ref{thm-main2LDCsep}  by combining this theorem with known bounds on the parameters of $2$-LDCs. 
One example of a $(2,\delta)$-LDC (over any field) is the Hadamard code, given by the list of all $m = 2^n$ zero-one vectors  $A = \{0,1\}^n$. This code has $\delta=1/2$ and is even in `special' form (the basis elements can be spanned as differences).  It is  known that one cannot construct $2$-LDCs that are much better than the Hadamard codes. This was first shown in \cite{GKST02} for small finite fields and then generalized to all fields in \cite{DviShp07}. We will prove here a slightly stronger bound for arbitrary fields, avoiding the constant factor loss inherent in \cite{DviShp07} (where $\delta$ is replaced by $\delta/2$). In fact, our proof is a direct adaptation of the information theoretic proof appearing in \cite{GKST02} and attributed to Alex Samorodnitsky.

\begin{thm}\label{thm-2LDClowerbound-new}
Let $\F$ be any field and let $A = (a_1,\ldots,a_m) \in (\F^n)^m$ be a special linear $(2,\delta)$-LDC. Then $$ m \geq 2^{2\delta n}.$$
\end{thm}

 Combining this bound with our generalization of  Efremenko's reduction we can prove our main theorem. 
\begin{proof}[Proof of Theorem~\ref{thm-main2LDCsep}]
Let $$ 	\rank(\rho(h) - \I ) = R>0. $$ By Theorem~\ref{thm-repLDC-2q} there exists a special linear $(2,\delta)$-LDC $A = (a_1,\ldots,a_m) \in (\F^t)^m$ with $m = |G|$, $ t \geq \frac{n}{R}$ and 
$\delta = \frac{\theta_\F \gamma_h}{2}.$ Applying Theorem~\ref{thm-2LDClowerbound-new} we get that
$$ m = |G| \geq 2^{2\delta t} \geq 2^{\theta_\F \gamma_h n/R},  $$ from which the bound follows.
\end{proof}

\paragraph{Organization:} We prove both Theorem~\ref{thm-rep2LDC} and Theorem~\ref{thm-repLDC-2q}  in Section~\ref{sec-rep2LDC}. We prove  Theorem~\ref{thm-2LDClowerbound-new} in Section~\ref{sec-LDC}. In Section~\ref{sec-3LDC} we discuss  the case of LDCs with  $q \geq 3$ and state a conjecture on representations motivated by conjectured bounds for these codes.

\paragraph{Acknowledgements:} I am grateful to Robert Guralnick and Klim Efremenko for helpful comments.

\section{From representations to LDCs}\label{sec-rep2LDC}

We begin with the proof of Theorem~\ref{thm-rep2LDC}. We treat all vectors $w \in \F^n$ as column vectors. For a matrix $A$ over $\F$ we denote by $A^t$ the transpose  of $A$. For a subspace $U \subset \F^n$ we denote by $U^\bot$ the subspace of all vectors $w \in \F^n$ such that the standard inner product  $u^t w = \ip{u}{w}=0$ for all $u \in U$. For a group element $g \in G $ and a subspace $U \subset \F^n$ we denote $U^g = \{\rho(g)u\,|\, u \in U\}$.

We will need the following easy claim which replaces the `dual basis' in Efremenko's construction (when $R=1$).

\begin{claim}\label{cla-choiceofbasis}
Let  $\rho : G \rightarrow \GL(\F^n)$ be an irreducible  representation of a finite group $G$. Let $U \subset \F^n$ be a subspace with $\dim(U)=k$. Then there exists $t \geq n/k$ group elements $g_1,\ldots,g_t \in G$ and $t$ vectors $w_1,\ldots,w_t \in \F^n$ such that $w_i \in \left(U^{g_j}\right)^\bot$ if and only if $i\neq j$.
\end{claim}
\begin{proof}
	Take $g_1,\ldots,g_t$ to be a minimal set of group elements for which 
	\begin{equation}\label{eq-sum-Ug-span}
		\sum_{j=1}^t U^{g_j}= \F^n.
	\end{equation}
	Such a set exists since the representation is irreducible and so the union of all of $U^g, g \in G$ spans $\F^n$. By minimality, we must have that for all $i \in [t]$
	\[ U^{g_i} \not\subset \sum_{j \neq i} U^{g_j}. \]
Hence, we can find, for each $i \in [t]$ a vector $w_i$ that is orthogonal to all the subspaces $\{  U^{g_j}\,|\, j\neq i \}$ and is not orthogonal to $U^{g_i}$.
The bound  $t \geq n/k$ stems from (\ref{eq-sum-Ug-span}) and the fact that $\dim(U^g)=k$ for all $g \in G$.
	\end{proof}

Let 
$$D = \sum_{\ell=1}^q \alpha_\ell \rho(h_\ell)$$ 
and suppose w.l.o.g that 
$$ \rank(D) = R \geq 1. $$
Thus, there exist two $n \times R$ matrices $X$ and $Y$ (both of rank $R$) such that 
$$ D = Y X^t.$$ 

Let $U \subset \F^n$ be the ($R$-dimensional) column span of $Y$. For $g \in G$ we denote by $$Y^g = \rho(g)Y$$ so that $U^g$ is the column space of $Y^g$.

We can now apply Claim~\ref{cla-choiceofbasis} to find group elements $g_1,\ldots,g_t \in G$ and vectors $w_1,\ldots,w_t \in \F^n$ with $t = \lceil n/R \rceil$ such that 
\begin{equation}\label{eq-wi_in_starj}
	w_i \in (U^{g_j})^\bot \,\,\textsf{iff}\,\, i\neq j.
\end{equation}
For each $j \in [t]$ define $\hat w_j$ to be 
\begin{equation}\label{eq-Ygj-times-wj}
	 0 \neq \hat w_j = (Y^{g_j})^t w_j \in \F^R
\end{equation}
($\hat w_j$ is non-zero because of (\ref{eq-wi_in_starj})). Let $W$ be the $n \times t$ matrix with columns $w_1,\ldots,w_t$. Then, from (\ref{eq-wi_in_starj}) and (\ref{eq-Ygj-times-wj}) we get that, for all $j \in [t]$ we have
\begin{equation}\label{eq-WstarYgj}
	W^t Y^{g_j} = \e_j \hat w_j^t,
\end{equation}
where $\e_j$ is the $j$'th standard basis (column) vector (note that the r.h.s $\e_j \hat w_j^t$ is a rank one $t \times R$ matrix). Indeed, the $i$'th row of the product $W^t Y^{g_j}$ will be zero for $i \neq j$ (since $w_i$ is orthogonal to the column span of $Y^{g_j}$ for $i \neq j$) and will be $\hat w_j$ for $i=j$. 

For  any  $z \in \F^n$, $j \in  [t]$, $s \in G$ we define
\begin{equation}\label{eq-beta_js}
	\beta_{j,s}(z) = \ip{X \hat w_j}{\rho(s)z}.
\end{equation}
Since $X$ is full rank and $\hat w_j \neq 0$, the set of $z$ for which $\beta_{j,s}(z)=0$ is a non trivial hyperplane. 

We are now ready to define the final LDC. Take $z \in \F^n$ to be determined later. For $s \in G$ we let 
$$ a_s = W^t \rho(s)z $$  and let 
$$A = (a_s)_{s \in G} \in (\F^t)^{|G|}$$
 (it is more natural in this case to index the elements of $A$ using $G$). We will now proceed to show that $A$ is an LDC (for some appropriate choice of $z$).

\begin{claim}\label{cla-spanning_q_tuples}
Let $A$ be defined above. Then, for every  $s \in G$ and any $j \in [t]$ we have
\begin{equation}
	\sum_{\ell=1}^q  \alpha_\ell a_{g_j h_\ell s} = \beta_{j,s}(z)  \e_j. 
\end{equation}	
\end{claim}
\begin{proof}
we proceed to calculate
\begin{eqnarray*}
	\sum_{\ell=1}^q  \alpha_\ell a_{g_j h_\ell s} &=& \sum_{\ell=1}^q \alpha_\ell W^t \rho(g_j h_\ell s)z \\
	&=& W^t \rho(g_j) \sum_{\ell=1}^q  \alpha_\ell \rho(h_\ell) \rho(s)z \\
	&=& W^t \rho(g_j) D \rho(s)z  \\
	&=& W^t \rho(g_j) Y X^t \rho(s) z \\
	&=& W^t Y^{g_j} X^t \rho(s) z. \\
	&=& \e_j \hat w_j^t X^t \rho(s) z \hspace{2cm} \text{By (\ref{eq-WstarYgj})}\\
	&=& \beta_{j,s}(z) \e_j.
\end{eqnarray*}	
\end{proof}

To show that $A$ is an LDC we need to construct, for each $j$, a large $q$-matching $M_j$ of tuples spanning $\e_j$. Claim~\ref{cla-spanning_q_tuples} tells us that, for any $j \in [t]$, we can include in $M_j$ any $q$-tuple of the form
$$T_{j,s} = \{ g_j h_1 s, \ldots, g_j h_q s \} $$   for which $\beta_{j,s}(z) \neq 0$. To find out the maximum number of disjoint tuples of this form we first consider the case when all the $\beta_{j,s}(z)$ are non-zero (which we can guarantee e.g. when $\F$ is infinite). We may  ignore the action of $g_j$ and determine the maximum number of disjoint tuples in the orbit of the tuple $(h_1,\ldots,h_q)$ under the right action of $G$. It is easy to see that one can get at least $|G|/q^2$ such disjoint $q$-tuples using the fact that each group element can be in at most $q$ distinct tuples. For each $j \in [t]$, let $M_j'$ denote such a $q$-matching  of size $|M_j'| \geq |G|/q^2$. Every tuple in $M_j'$ is associated with some hyperplane $\beta_{j,s}(z)=0$ and we can only keep the tuples for which $\beta_{j,s}(z) \neq 0$. If $\F$ is finite, a random choice of $z \in \F^n$ will avoid each hyperplane with probability $1 - 1/|\F|$ and so, using expectation, there exists some $z \in \F^n$ for which we can keep at least a  $1 - 1/|\F|$ fraction of the tuples in $\cup_{j}M_j'$.  Picking such a $z$, and letting, for each $j$, $M_j$ be the $q$-matching obtained from $M_j'$ by throwing away all tuples with $\beta_{j,s}(z)=0$, we get

\begin{equation}
	\sum_{j \in [t]} |M_j| \geq \left(1 - \frac{1}{|\F|}\right)\sum_{j \in [t]} |M_j'| \geq \left(1 - \frac{1}{|\F|}\right) \cdot \frac{1}{q^2}\cdot  t \cdot |G|.
\end{equation}

Hence, $A$ is  a $(q,\delta)$-LDC with $\delta \geq \theta_\F/q^2$.  This completes the proof of Theorem~\ref{thm-rep2LDC}.

\paragraph{Proof of Theorem~\ref{thm-repLDC-2q}:}
In the case $h_1 = h$, $h_2 = \id$ (and $\alpha_1 = 1$, $\alpha_2 = -1$)  Claim~\ref{cla-spanning_q_tuples} says that for all $j \in [t]$ and $s \in G$ we have
$$a_{g_j h s } - a_{g_j s} = \beta_{j,s}(z) \e_j.$$
 Hence, it is enough to construct many disjoint pairs of the form $(g_j h s, g_j s)$ with $\beta_{j,s}(z) \neq 0$. Assume first that all $\beta_{j,s}(z)$ are non-zero. Dropping the $g_j$'s as before, we observe that the graph with edges $\{(hs,s) | s \in G\}$ is a disjoint union of cycles of length $\ord(h)$ (or a perfect matching, if $\ord(h)=2$). If $\ord(h)$ is even then this graph has a perfect matching, which gives $\delta=1/2$. If $\ord(h)$ is odd, we can pick $(\ord(h)-1)/2$ disjoint edges in each cycle, which gives $\delta = 1/2 - 1/2\ord(h)$ as required. The argument for finite $\F$ is the same as before, and loses another factor of $\theta_\F$.

\begin{com}\label{com-lambda}
	If we start with $\rho(h) - \lambda \I$ the proof is essentially the same. However, we end up with a $2$-LDC that is not quite in `special' form as the pairs in $M_j$ span $e_j$ with the help of $\lambda$. That is, for every pair $(g,g') \in M_j$ we have $a_g - \lambda a_{g'} = \beta \cdot e_j$ for some non-zero $\beta$. This can be turned into special form LDC by adding to the sequence $A$ the sequence $\lambda A$ of all $\lambda a_g, g \in G$. The sizes of the matchings  stay the same but the length of the code doubles to $2|G|$ and so $\delta$ becomes $\delta/2$. Applying Theorem~\ref{thm-2LDClowerbound-new} with these parameters we get the bound stated in (\ref{eq-rankwithlambda}).
\end{com}

\section{Lower bounds for $2$-LDCs over arbitrary fields}\label{sec-LDC}

In this section we prove Theorem~\ref{thm-2LDClowerbound-new}. As was mentioned in the introduction, this is essentially the proof appearing in \cite{GKST02} for the binary case and is attributed to Alex Samorodnitsky. We simply observe that one can apply  the same proof steps also for larger alphabets.
We start with preliminaries on the entropy function.
For a distribution $$\mu = (\mu_1,\ldots,\mu_n) \in [0,1]^n, \sum_{i}\mu_i = 1$$ 
we let 
$$ H(\mu) = \sum_{i, \mu_i\neq 0} \mu_i \log_2(1/\mu_i) $$ be the (Shannon) entropy function. If $X$ is a random variable with finite support, we use $H(X)$ to denote the entropy of the distribution of $X$. For two random variables $X$ and $Y$ we  define the conditional entropy of $X$ given $Y$ to be $$ H(X|Y) = \sum_{y \in \supp(Y)}\Pr[Y = y]H(X|Y=y),$$
where $H(X|Y=y)$ is the entropy of $X$ conditioned on the event  $Y=y$.

We will use the following basic properties of the entropy function (see any introductory text on information theory, e.g. \cite{CovTho06}).
\begin{claim}\label{cla-entropy}
The entropy function $H(\cdot)$ satisfies the following:
\begin{enumerate}
	\item If $X$ is a random variable taking at most $m$ values then $H(X) \leq \log_2(m)$.
	\item For every $X$ and $Y$ we have $$ H(X) \geq H(X|Y).$$
	\item For every random variable $X$ and any event $E$ we have that 
$$ H(X) \geq \Pr[E] H(X|E). $$
	\item Let $X = (X_1,\ldots,X_n)$ be a random variable distributed over  $\Sigma^n$ with $\Sigma$ a finite set. Then
 $$ H(X) = \sum_{i=1}^{n} H(X_i | X_{1,\ldots,i-1}),$$ where $X_{1,\ldots,i-1} = (X_1,\ldots,X_{i-1})$ is the prefix of $X$.
 
\end{enumerate}
\end{claim}

The following claim is the only place where we have to argue about a larger alphabet more carefully.
\begin{claim}\label{cla-match-entropy}
	Let $X$ be a random variable distributed uniformly over $[t]$ and let $M$ be a collection of $s$ disjoint pairs over $[t]$ (i.e., a $2$-matching of size $s$). Let $f : [t] \rightarrow T$ be any function to some finite set $T
$. Suppose that for every pair $\{j_1,j_2\} \in M$ we have $f(j_1) \neq f(j_2)$. Then $$ H(f(X)) \geq \frac{2s}{t}.$$
\end{claim}
\begin{proof}
	W.l.o.g we may assume that
	$$ M = \{ \{2i-1,2i\} \,|\, i=1,\ldots,s\}.$$
	Let us first consider the case when $t = 2s$ and so $M$ is a perfect matching. Let 
	$$ Y = \left\lceil \frac{X}{2}\right\rceil$$
	denote the index $i$ of the matching pair in which $X$ lands. For every $i \in [s]$, conditioned on the event $Y=i$, $f(X)$ takes two different values with equal probability and so 
	$$ H(f(X)| Y=i) = H(1/2,1/2) =  1. $$
	We now have
	$$ H(f(X)) \geq H(f(X)|Y) =  \sum_{i=1}^s \Pr[Y = i] H(f(X) | Y = i) = 1, $$ 
	as was required (for the case $t = 2s$).
	
	Next suppose $2s < t$ and let $E$ be the event 
	 that $X$ lands in one of the pairs in $M$ :
	 $$ E  = \{ X \leq 2s \}. $$
	 From the case $t = 2s$ we know that 
	 $$ H(f(X)|E) \geq 1.$$  
	 It follows that
	 $$ H(f(X)) \geq \Pr[E] H(f(X)|E) \geq \frac{2s}{t}$$ and so we are done.
\end{proof}

\subsection{Proof of Theorem~\ref{thm-2LDClowerbound-new}}
	Let 
	$$X = (X_1,\ldots,X_n)$$ 
	be a random variable uniformly distributed over the sequence $A$.  More formally let $K$ be a uniform random variable over $[m]$ and let $X = a_K$.
	
	By Item 1 of Claim~\ref{cla-entropy} we have
	\begin{equation}\label{eq-entropy-logupper}
		H(X) \leq \log_2(m).
	\end{equation}
	By Item 4 of Claim~\ref{cla-entropy} we have
	\begin{equation}\label{eq-entropy-sumcond}
		H(X) = \sum_{i=1}^{n} H(X_i | X_{1,\ldots,i-1}).
	\end{equation}
	We will now proceed to bound (\ref{eq-entropy-sumcond}) from below term-by-term. We first set up some notations. For each  $i \in [n]$ and  $b \in \supp(X_{1,\ldots,i-1})$,  let 
	$$ J_i^b = \{ j \in [m] \,|\, (a_j)_{1,\ldots,i-1} = b\}$$ 
	be the (non empty) set of indices in which $b$ appears as a prefix (for $i=1$ we have only one set $J_1^{\emptyset} = [m]$). Notice that the matching $M_i$ respects the partition induced by the sets $J_i^b$ since a pair in $M_i$ must differ only in the $i$'th coordinate. We can thus partition the matching $M_i$ into matchings 
	$$ M_i^b = \{ \{j_1,j_2\} \in M_i \,|\, j_1,j_2 \in J_i^b\}$$
	So that 
	\begin{equation}
		\sum_{b}|M_i^b| = |M_i|.
	\end{equation}
	
Now, from Claim~\ref{cla-match-entropy} we get that for every $i \in [n]$ and every $b \in \supp(X_{1,\ldots,i-1})$ it holds that
\begin{equation}
	H(X_i | X_{1,\ldots,i-1}=b) \geq \frac{2|M_i^b|}{|J_i^b|}.
\end{equation}
Therefore,
\begin{eqnarray*}
	H(X_i|X_{1,\ldots,i-1}) &=& \sum_{b \in \supp(X_{1,\ldots,i-1})} \Pr[ X_{1,\ldots,i-1}= b ] H(X_i | X_{1,\ldots,i-1}= b )\\
	&\geq &  \sum_{b \in \supp(X_{1,\ldots,i-1})} \frac{|J_i^b|}{m} \cdot \frac{2|M_i^b|}{|J_i^b|}.\\
	&=& \frac{2|M_i|}{m} 
\end{eqnarray*}
Plugging this into (\ref{eq-entropy-sumcond}) we get 
$$ H(X) \geq \sum_{i=1}^n \frac{2|M_i|}{m} \geq  2\delta n, $$ 
which, using (\ref{eq-entropy-logupper}) implies the desired bound. \qed

\section{Using LDC lower bounds for $q \geq 3$:}\label{sec-3LDC}

 It is natural to try and use lower bounds for general $q$-LDC to show that $q$ matrices of the form $\rho(g)$ cannot span a low rank matrix. Unfortunately, the bounds known for $q \geq 3$ are much weaker and, in general, not believed to be tight. In particular, for large fields, the best lower bounds known for $q \geq 3$ are at best quadratic $m \geq n^2$. Combining such a bound with Theorem~\ref{thm-rep2LDC} will not give anything useful since we know that any irreducible representation has dimension at most $|G|^{1/2}$ and hence the LDC obtained from Theorem~\ref{thm-rep2LDC} is already quadratic at best (even for $R=1$). Over small fields there is a cubic lower bound known for $q=3$.\footnote{ The bound holds also for non linear LDCs but we will state it only for linear since we haven't defined non linear codes.}

\begin{thm}[\cite{AGVK23}]\label{thm-3LDC}
Let $\F$ be a finite field. Suppose $A = (a_1,\ldots,a_m) \in (\F^n)^m$ is a $(3,\delta)$-LDC. Then $$ n^3 \leq C(\F,\delta) \cdot m \log_2^6(m),$$ where $C(\F,\delta)$ is a constant depending only on $\delta$ and $|\F|$ (the dependence is polynomial).
\end{thm}

Combining this bound with Theorem~\ref{thm-rep2LDC}, we get the following very weak form of Theorem~\ref{thm-main2LDCsep} for $q=3$.
\begin{thm}\label{thm-3LDC}
Let $\F$ be a finite field and let $\rho : G \rightarrow \GL(\F^n)$ be an irreducible representation of a finite group $G$. Let $h_1,h_2,h_3 \in G$ be group elements and $\alpha_1,\alpha_2,\alpha_3 \in \F$ be scalars such that $$\alpha_1 \rho(h_1) + \alpha_2 \rho( h_2) + \alpha_3 \rho(h_3) \neq 0.$$ Then
$$ \rank( \alpha_1 \rho(h_1) + \alpha_2 \rho( h_2) + \alpha_3 \rho(h_3) ) \geq \frac{n}{c(\F) |G|^{1/3}\log^2_2|G|},$$ with $c(\F)$ a constant depending only on $|\F|$.
\end{thm}

It is widely believed that there are no $q$-LDCs with constant $q$ and with $m  = \poly(n)$. The best known constructions \cite{Yekhanin08,Efremenko09} give $m < 2^{n^{o(1)}}$ but are far from polynomial. It is also true that all known constructions of LDCs can be derived from low rank (in fact, rank one) matrices spanned by $q$ group elements (see \cite{Efremenko09}).  We make the following conjecture, which would follow from super-polynomial lower bounds on constant query LDCs.\begin{conjecture}
	For every $\eps>0$ and $q \in \N$ there is $N_{\eps,q} \in \N$ such that the following holds for all $n > N_{\eps,q}$:
	
	Let $\rho : G \rightarrow \GL(\F^n)$ be an irreducible representation of a finite group $G$ such that $n \geq |G|^\eps$. Let $h_1,\ldots,h_q \in G$ and let $\alpha_1,\ldots,\alpha_q \in \F$ be such that 
	$$ \sum_{\ell=1}^q \alpha_\ell \rho(h_\ell) \neq 0.$$ Then
	$$ \rank\left(\sum_{\ell=1}^q \alpha_\ell \rho(h_\ell) \right) \geq n^{1-\eps}$$
\end{conjecture}

Indeed, by Theorem~\ref{thm-rep2LDC}, a counter example to this conjecture would give us some fixed $q$ and $C$ and an infinite family of $(q,1/q^2)$-LDCs with $m \leq  n^C$  (take $C = 1/\eps^2$).

\bibliographystyle{alpha}

\bibliography{IrrepSeparate}

\end{document}